\documentclass[12pt]{amsart}
\usepackage{mathrsfs}
\usepackage{amssymb} 
\usepackage{url}
\usepackage{hyperref}
\usepackage{color}
\definecolor{refcolor}{RGB}{0,0,190}
\hypersetup{
    colorlinks,
    citecolor=refcolor,
    filecolor=refcolor,
    linkcolor=refcolor,
    urlcolor=refcolor
}

\theoremstyle{definition}
\newtheorem{theorem}{Theorem}[section]
\newtheorem{definition}[theorem]{Definition}

\newtheorem{lemma}[theorem]{Lemma}
\newtheorem{corollary}[theorem]{Corollary}
\newtheorem{remark}[theorem]{Remark}

\def\({\left(}
\def\){\right)}

\newcommand{\R}{\mathbb{R}}

\newcommand{\de}{\textnormal{d}}

\newcommand{\tn}{\textnormal}
\newcommand{\ds}{\displaystyle}

\newcommand{\ie}{\textit{i.e.} }
\newcommand{\cf}{\textit{cf.} }
\newcommand{\eg}{\textit{e.g.} }

\newcommand{\citep}[2]{\cite{#1}, p. #2}

\newcommand{\cfeg}[2]{(\cf \eg \citep{#1}{#2})}

\newcommand{\rank}{\textnormal{rank }}
\newcommand{\diag}{\textnormal{diag}}

\newcommand{\mf}[1]{\mathfrak{#1}}
\newcommand{\mc}[1]{\mathcal{#1}}
\newcommand{\ms}[1]{\mathscr{#1}}

\newcommand{\sref}[1]{\S\ref{#1}}

\newcommand{\idxannih}[2]{#1{}^{#2}{}}
\newcommand{\idxcoannih}[2]{#1{}_{#2}{}}

\newcommand{\annih}[1]{\idxannih{#1}{\bullet}}
\newcommand{\coannih}[1]{\idxcoannih{#1}{\bullet}}

\newcommand{\annihg}{\coannih{g}}

\newcommand{\metric}[1]{\langle#1\rangle}

\newcommand{\cocontr}{{{}_\bullet}}

\newcommand{\vectmodule}{\mf X}
\newcommand{\fivect}[1]{\vectmodule(#1)}

\newcommand{\fiscal}[1]{\ms F(#1)}

\newcommand{\annihforms}[1]{\annih{\mc A}(#1)}

\newcommand{\lie}{\mc L}
\newcommand{\kosz}{\mc K}

\newcommand{\metricname}{fundamental tensor}

\def\hyph{-\penalty0\hskip0pt\relax}
\newcommand{\semiriem}{semi{\hyph}Riemannian}

\newcommand{\semireg}{semi{\hyph}regular}

\newcommand{\nondeg}{non{\hyph}degenerate}

\newcommand{\rstationary}{radical{\hyph}stationary}

\begin{document} 
 
\title{Cartan's Structural Equations for Singular Manifolds}
\author{Cristi \ Stoica}
\date{\today}
\thanks{Partially supported by Romanian Government grant PN II Idei 1187.}

\begin{abstract}
The classical Cartan's structural equations show in a compact way the relation between a connection and its curvature, and reveals their geometric interpretation in terms of moving frames. In order to study the mathematical properties of singularities, we need to study the geometry of manifolds endowed on the tangent bundle with a symmetric bilinear form which is allowed to become degenerate. But if the fundamental tensor is allowed to be degenerate, there are some obstructions in constructing the geometric objects normally associated to the fundamental tensor. Also, local orthonormal frames and coframes no longer exist, as well as the metric connection and its curvature operator. This article shows that, if the fundamental tensor is radical stationary, we can construct in a canonical way geometric objects, determined only by the fundamental form, similar to the connection and curvature forms of Cartan. In particular, if the fundamental tensor is non-degenerate, we obtain the usual connection and curvature forms of Cartan. We write analogs of Cartan's first and second structural equations. As a byproduct we will find a compact version of the Koszul formula.
\bigskip
\noindent 
\end{abstract}


\maketitle

\setcounter{tocdepth}{1}
\tableofcontents

%
\section{Introduction}

In {\semiriem} geometry (including Riemannian), Cartan's first and second structural equations establish the relation between a local orthonormal frame, the connection, and its curvature. We are interested in having such powerful tools in {\em singular {\semiriem} geometry} \cite{Sto11a,Sto12e}, which is the geometry of manifolds endowed on the tangent bundle with a symmetric bilinear form, which is allowed to become degenerate and change the signature. Singular {\semiriem} geometry has been used successfully by the author to study the big bang singularity of the Friedmann-Lema\^itre-Robertson-Walker spacetimes \cite{Sto12a,Sto12c}, the black hole singularities of the Schwarzschild, Reissner-Nordstr\"om, and Kerr-Newman spacetimes \cite{Sto11e,Sto11f,Sto11g,Sto14a}, the Einstein equation at singularities \cite{Sto12b} and quantum gravity \cite{Sto12d}.

But if we replace the metric with a symmetric bilinear form which is allowed to become degenerate, as in singular {\semiriem} geometry, the {\metricname} cannot be inverted, to construct orthonormal coframes. Moreover, in the absence of the inverse of the {\metricname}, we need to find a way to construct objects which do the job of the Levi-Civita connection, and of the Riemann curvature. One important operation is the metric contraction between covariant indices, which requires a contravariant {\metricname}. In Riemannian and {\semiriem} geometry, the contravariant {\metricname} is obtained by inverting the covariant {\metricname} \cite{Balan1999diffgeom,udriste2005linear}, but if this one is degenerate, we can no longer invert it.

These difficulties were avoided in \cite{Sto11a}, where instead of the metric connection was used the Koszul object, and it was defined a Riemann curvature $R_{abcd}$, which coincides to the usual Riemann curvature tensor if the {\metricname} is {\nondeg}. In \cite{Sto11a} it was shown that, even when the metric is not invertible, a generalized metric contraction at a point $p\in M$ can be defined in an invariant way, on the subspace $\annih{T}_pM$ of the cotangent space $T_p^*M$, which consists in covectors of the form $\omega(V)=\metric{U,V}$, $U,V\in T_p M$. The contraction was shown to be well defined and has been extended to tensors of higher order, so long as these tensors live in the subspace $\annih{T}_pM$. This contraction was used to define the Riemann curvature tensor $R_{abcd}$.

In this article, I will show how to write structural equations similar to Cartan's, but which are valid for singular {\semiriem} geometry, as well as for the non-singular one. Section \sref{ss_singular_semi_riem} recalls briefly the main notions of singular {\semiriem} manifolds, which will be used in the article. In section \sref{s_cartan_structure_i} I construct the connection forms, and derive the first structural equation for {\rstationary} manifolds. The curvature forms are defined in section \sref{s_cartan_structure_ii}, which contains the derivation of the second structural equation for {\rstationary} manifolds.

\section{Brief review of singular {\semiriem} manifolds}
\label{ss_singular_semi_riem}

We recall some of the main results about singular {\semiriem} manifolds, from \cite{Sto11a}.

This paper is concerned with differentiable manifolds $M$, endowed with a symmetric bilinear form $g\in \Gamma(T^*M \odot_M T^*M)$, named \textit{{\metricname}} or \textit{metric}. The {\metricname} is allowed to be degenerate, or to become degenerate on some regions. In this case, the pair $(M,g)$ is named \textit{singular {\semiriem} manifold}.

\begin{remark}
The name ``singular {\semiriem} manifold'' may be not very inspired, since it suggests that such a manifold is {\semiriem}, while in fact is more general, containing the {\nondeg} case as a subcase. Despite this inadvertence, this name is generally used in the literature (\cite{Moi40,Vra42,Kup96}), and I adhere to it. In addition, whenever I introduce geometric objects which are similar to objects from the {\nondeg} {\semiriem} geometry, and which generalize them, I try as much as possible to use the standard terminology from {\semiriem} geometry (see \eg \cite{ONe83}).
\end{remark}

At any point $p\in M$, there is a frame in $T_pM$ in which the fundamental tensor $g$ has the diagonal form
$$\diag(0,\ldots,0,-1,\ldots,-1,1,\ldots,1),$$
in which $0$ appears $r$ times, $-1$ appears $s$ times, and $1$ appears $t$ times. The triple $(r,s,t)$ is named the \textit{signature} of of $g$, and $\dim M=r+s+t$, and $\rank g=s+t=n-r$. If the signature is allowed to vary from point to point, $(M,g)$ is said to be with \textit{variable signature}, otherwise it is said to be with \textit{constant signature}. If $g$ is {\nondeg}, then $(M,g)$ is named \textit{{\semiriem} manifold}, and if in addition it is positive definite, $(M,g)$ is named \textit{Riemannian manifold}.

\begin{remark}
The theory developed in \cite{Sto11a}, and here, does not make any assumptions about the degeneracy of the {\metricname}. Because of this, these results also apply to Riemannian and {\semiriem} manifolds.
\end{remark}

\begin{remark}
The signature of the {\metricname} of a singular {\semiriem} manifold may vary from one region to another. In \cite{Kup87b,Kup96}, and in general in the literature, was preferred to maintain the signature constant. This is justifiable, because the most singular behavior of fields takes place at the boundary between two regions of different signature. But the development which took place in \cite{Sto11a}, especially the introduction of the {\semireg} manifolds, allows us to deal also with the situations when the signature of the {\metricname} is allowed to change.
\end{remark}

The remaining of this section recalls very briefly the main notions and results on singular {\semiriem} manifolds, as developed in \cite{Sto11a}.

Let's define the operator $\flat:T_pM\to T^*_pM$, which associates to each vector $X_p\in T_pM$ the $1$-form $\flat(X_p)=X_p^{\flat}\in T^*_pM$, defined as $X_p^{\flat}(Y):=\metric{X_p,Y_p}$, for any $Y_p\in T_pM$. The operator $\flat$ is a morphism of vector spaces; it is an isomorphism if and only if the {\metricname} is {\nondeg}, but the fact that it is a morphism will be enough for the following. Note that, in general, there is no corresponding inverse operator $\sharp=\flat^{-1}$.

Let $\annih{(T_pM)} = \flat(T_pM) \subseteq T^*_pM$ be the space of covectors at $p\in M$ which can be expressed as $\omega_p=Y_p^{\flat}$, for some $Y_p\in T_p M$. We denote by $\annih{T}M$ the subset of the cotangent bundle defined as
\begin{equation}
	\annih{T}M=\bigcup_{p\in M}\annih{(T_pM)}.
\end{equation}
 The space $\annih{T}M$ is a vector bundle if and only if the signature of the {\metricname} is constant.

We can define the subset of sections of $T^*M$ which are valued, at each $p$, in $\annih{(T_pM)}$, by
\begin{equation}
	\annihforms{M}:=\{\omega\in\Gamma(T^*M)|\omega_p\in\annih{(T_pM)}\tn{ for any }p\in M\}.
\end{equation}

On $\annih{T_p}M$ we can define a unique {\nondeg} inner product $\annihg_p$ by
$$\annihg_p(\omega_p,\tau_p):=\metric{X_p,Y_p},$$
where $X_p,Y_p\in T_p M$, $X_p^{\flat}=\omega_p$ and $Y_p^{\flat}=\tau_p$. We alternatively use the notations $\annihg(\omega_p,\tau_p)=\omega_p(\cocontr)\tau_p(\cocontr)$, and call $\omega_p(\cocontr)\tau_p(\cocontr)$ the \textit{covariant metric contraction} between $\omega_p$ and $\tau_p$. 
This metric contraction is defined at all points of the manifold $M$, and there is no need to exclude any of them. It is invariant for components living in the space $\annih{(T_pM)}$. It is not defined for components living in $T_pM^*-\annih{(T_pM)}$, therefore it will be applied only to components living in the space $\annih{(T_pM)}$.

\textit{The Koszul object} is defined as
\begin{equation}
\label{eq_Koszul_form}
	\kosz:\fivect M^3\to\R,
\end{equation}
\begin{equation*}
\begin{array}{llll}
	\kosz(X,Y,Z) &:=&\ds{\frac 1 2} \{ X \metric{Y,Z} + Y \metric{Z,X} - Z \metric{X,Y} \\
	&&\ - \metric{X,[Y,Z]} + \metric{Y, [Z,X]} + \metric{Z, [X,Y]}\}.
\end{array}
\end{equation*}

For a non-degenerate {\metricname} $g$, we can define the covariant derivative $\nabla_X Y$ of a vector field $Y$ in the direction of a vector field $X$, where $X,Y\in\fivect{M}$, by the \textit{Koszul formula} $\metric{\nabla_X Y,Z}=\kosz(X,Y,Z)$ \cfeg{ONe83}{61}. For a {\metricname} $g$ which can be degenerate, the covariant derivative cannot be extracted from the Koszul formula. But a lot of what can be done using the covariant derivative, can also be accomplished by working with the Koszul object, which is defined and smooth independently on the rank and signature of $g$.

Like the Koszul object from the {\nondeg} case, our Koszul object is not a tensor, but it can be used to construct a Riemann curvature, which is tensor. Its components in a chart are $\kosz_{abc}=\kosz(\partial_a,\partial_b,\partial_c)=\frac 1 2 (\partial_a g_{bc} + \partial_b g_{ca} - \partial_c g_{ab})$, which are Christoffel's symbols of the first kind, $\Gamma_{cab}=[ab,c]$.

\begin{definition}[see \cite{Kup96} Definition 3.1.3]
\label{def_radical_stationary_manifold}
A manifold $(M,g)$ is called \textit{{\rstationary}} if the Koszul object satisfies the condition 
\begin{equation}
\label{eq_radical_stationary_manifold}
		\kosz(X,Y,\_)\in\annihforms M,
\end{equation}
for any $X,Y\in\fivect{M}$.
\end{definition}

\begin{definition}
\label{def_semi_regular_semi_riemannian}
A {\rstationary} manifold $(M,g)$ whose Koszul object satisfies
\begin{equation}
	\kosz(X,Y,\cocontr)\kosz(Z,T,\cocontr) \in \fiscal M
\end{equation}
for any $X,Y,Z,T\in\fivect M$, is called \textit{{\semireg} manifold}.
\end{definition}

\begin{remark}
If the signature of the {\metricname} of a {\rstationary} manifold $(M,g)$ is constant, then $(M,g)$ is {\semireg}. The important difference appears at the points where the signature of $g$ changes. In general, at signature changes the covariant metric contraction blows up. The condition from the definition \ref{def_semi_regular_semi_riemannian} ensures that this doesn't happen for $\kosz(X,Y,\cocontr)\kosz(Z,T,\cocontr)$. This is enough to ensure nice properties, in particular a smooth Riemann curvature (definition \ref{def_riemann_curvature}).
\end{remark}

\begin{remark}
In \cite{Sto11a} it was given a different definition for {\semireg} manifolds. The Definition \ref{def_semi_regular_semi_riemannian} was proved in \cite{Sto11a} to be equivalent. Similarly, the definition of the Riemann curvature given in \cite{Sto11a} is different than the one given below, but they are shown to be equivalent. I preferred here these definitions, because they simplify the task of finding the structure equations.
\end{remark}

\begin{definition}
\label{def_riemann_curvature}
Let $(M,g)$ be a {\rstationary} manifold. The object
\begin{equation}
\label{eq_riemann_curvature_tensor_koszul_formula}
	R: \fivect M\times \fivect M\times \fivect M\times \fivect M \to \R,
\end{equation}
\begin{equation*}
\begin{array}{lll}
	R(X,Y,Z,T)&=& X \kosz(Y,Z,T) - Y \kosz(X,Z,T) - \kosz([X,Y],Z,T)\\
	&& + \kosz(X,Z,\cocontr)\kosz(Y,T,\cocontr) - \kosz(Y,Z,\cocontr)\kosz(X,T,\cocontr),
\end{array}
\end{equation*}
for any vector fields $X,Y,Z,T\in\fivect{M}$, is called the \textit{Riemann curvature} associated to the Koszul object $\kosz$.
\end{definition}

\begin{remark}
In \cite{Sto11a} is shown that the Riemann curvature is a tensor, and has the same symmetries as the Riemann curvature tensor from the {\nondeg} case. In the {\nondeg} case, it coincides to the Riemann curvature tensor $R(X,Y,Z,T)$.
\end{remark}

\section{The first structural equation}
\label{s_cartan_structure_i}

Cartan's first structural equation shows how a moving coframe rotates when moving in one direction, due to the connection. In the following, we will derive the first structural equation for the case when the {\metricname} is allowed to be degenerate. Of course, in this case we will not have a notion of local orthonormal frame, and we will work instead with vectors and annihilator covectors.
The following decomposition of the Koszul object will be needed to derive the first structural equation.

\subsection{The decomposition of the Koszul object}
\label{s_koszul_form_compact}

\begin{lemma}
\label{thm_Koszul_form_compact}
The Koszul object \eqref{eq_Koszul_form}
 decomposes as
\begin{equation}
\label{eq_Koszul_form_compact}
	2\kosz(X,Y,Z) = (\de Y^{\flat})(X, Z) + (\lie_Y g)(X,Z).
\end{equation}
\end{lemma}
\begin{proof}
From the formula for the exterior derivative we get
\begin{equation*}
\begin{array}{lll}
(\de Y^{\flat})(X, Z) &=& X\(Y^{\flat}(Z)\) - Z\(Y^{\flat}(X)\) - Y^{\flat}([X,Z]) \\
&=&  X\metric{Y,Z} - Z\metric{X,Y} + \metric{Y,[Z,X]}. \\
\end{array}
\end{equation*}
The Lie derivative is
\begin{equation*}
\begin{array}{lll}
	(\lie_Y g)(Z,X) &=& Y g(Z,X) - g([Y,Z],X) - g(Z,[Y,X]) \\
	&=& Y \metric{Z,X} - \metric{X,[Y,Z]} + \metric{Z,[X,Y]}. \\
\end{array}
\end{equation*}
The relation \eqref{eq_Koszul_form_compact} follows then immediately.
\end{proof}

\begin{corollary} The Koszul object \eqref{eq_Koszul_form} has the property
\begin{equation}
(\de Y^{\flat})(X, Z) =	\kosz(X,Y,Z) - \kosz(Z,Y,X) .
\end{equation}
\end{corollary}
\begin{proof}
This is an immediate consequence of the properties of the Koszul object and the Lemma \ref{thm_Koszul_form_compact}.
\end{proof}

\subsection{The connection forms}
\label{s_conn_form}

Let $(M,g)$ be a {\nondeg} manifold. If $(E_a)_{a=1}^n$ is a local orthonormal frame on $M$ with respect to $g$, then its dual $(\omega^b)_{b=1}^n$, defined by $\omega^b(E_a)=\delta^b_a$, is also orthonormal. 
The $1$-forms $\omega_a{}^b$, $1\leq a,b\leq n$ defined as
\begin{equation}
\label{eq_connection_forms_nondeg}
	\omega_a{}^b(X) := \omega^b(\nabla_X E_a)
\end{equation}
are called the \textit{connection forms} (\cf \eg \cite{ONe95}).

It is important to be aware that the indices $a,b$ label the connection $1$-forms $\omega_a{}^b$,  and they don't represent the components of a form.

For a general {\metricname} $g$, there is no Levi-Civita connection $\nabla_X$ with respect to $g$, and hence $\nabla_X E_a$ does not exist. Also, a frame $(E_a)_{a=1}^n$ cannot be orthonormal with respect to $g$, but it can be orthogonal. But its dual frame $(\omega^b)_{b=1}^n$ cannot even be orthogonal, because the {\metricname} $\annihg(\omega,\tau)$ is not defined for the entire $T^*M$, but only for $\annih{T}M$. Here I show an alternative way to define connection $1$-forms for this case.

\begin{definition}
\label{def_connection_form}
Let $X,Y\in\fivect{M}$ be two vector fields. Then, the $1$-form defined as
\begin{equation}
\label{eq_connection_forms}
	\omega_{XY}(Z) := \kosz(Z,X,Y)
\end{equation}
is called the \textit{connection form} associated to the Koszul object $\kosz$ and the vector fields $X,Y$.
In particular, we define $\omega_{ab}$ by
\begin{equation}
\label{eq_connection_forms_frame}
	\omega_{ab}(Z) := \omega_{E_a E_b}(Z).
\end{equation}
\end{definition}

\begin{remark}
The fact that $\omega_{XY}$ is $1$-form follows from the fact that the Koszul object is linear, and $\fiscal M$-linear in the first argument.
\end{remark}

\subsection{The first structural equation}
\label{ss_cartan_structure_i}

Let $(M,g)$ be a {\rstationary} manifold.

\begin{lemma}
\label{thm_cartan_structure_i}
The following equation, called \textit{the first structural equation} determined by the metric $g$, holds
\begin{equation}
\label{eq_cartan_structure_i}
	\de X^{\flat} = \omega_{X\cocontr} \wedge \cocontr^{\flat},
\end{equation}
where $\omega_{X\cocontr} \wedge \cocontr^{\flat}$ is the metric contraction of $\omega_{XY}\wedge Z^{\flat}$ in $Y,Z$.
\end{lemma}
\begin{proof}
From the Lemma \ref{thm_Koszul_form_compact} and from the formula \eqref{eq_Koszul_form} of the Koszul object, we find
\begin{equation}
\label{eq_cartan_structure_i_kosz}
	(\de X^{\flat})(Y,Z) = \kosz(Y,X,Z) - \kosz(Z,X,Y).
\end{equation}
By replacing the Koszul object with the connection $1$-form, we get
\begin{equation}
\label{eq_cartan_structure_i_expanded}
	(\de X^{\flat})(Y,Z) = \omega_{XZ}(Y) - \omega_{XY}(Z).
\end{equation}
By using the properties of the metric contraction and the property of $(M,g)$ of being {\rstationary}, we can expand the Koszul object as
\begin{equation}
\kosz(X,Y,Z) = \kosz(X,Y,\cocontr)\metric{\cocontr,Z} = \kosz(X,Y,\cocontr)\(\cocontr^{\flat}(Z)\).
\end{equation}
We can do the same for the connection $1$-form, \ie,
\begin{equation}
\begin{array}{lll}
\omega_{YZ}(X) &=& \omega_{Y\cocontr}(X)\metric{\cocontr,Z} = \omega_{Y\cocontr}(X)\(\cocontr^{\flat}(Z)\) \\
&=& \(\omega_{Y\cocontr}\otimes\cocontr^{\flat}\)(X,Z).
\end{array}
\end{equation}
The equation \eqref{eq_cartan_structure_i_expanded} becomes
\begin{equation}
\begin{array}{lll}
	(\de X^{\flat})(Y,Z) &=& \(\omega_{X\cocontr}\otimes\cocontr^{\flat}\)(Y,Z) - \(\omega_{X\cocontr}\otimes\cocontr^{\flat}\)(Z,Y) \\
	&=& \(\omega_{X\cocontr} \wedge \cocontr^{\flat}\)(Y,Z).
\end{array}
\end{equation}
\end{proof}

The following corollary shows how we get the first structural equation as we know it.
\begin{corollary}
\label{thm_cartan_structure_i_std}
If the {\metricname} $g$ is {\nondeg}, $(E_a)_{a=1}^n$ is an orthonormal frame, and $(\omega^a)_{a=1}^n$ is its dual, then
\begin{equation}
\label{eq_cartan_structure_i_std}
	\de \omega^a = -\omega_s{}^a \wedge \omega^s.
\end{equation}
\end{corollary}
\begin{proof}
From $\kosz(X,Y,Z) + \kosz(X,Z,Y) = X \metric{Y,Z}$, it follows
\begin{equation}
	\omega_{E_a E_b}(X) + \omega_{E_b E_a}(X) = X\metric{E_a,E_b} = X(\delta_{ab}) = 0,
\end{equation}
and therefore
\begin{equation}
	\omega_{E_a E_b} = -\omega_{E_b E_a}.
\end{equation}
From equation \eqref{eq_cartan_structure_i} we obtain
\begin{equation}
	\de E_a^{\flat} = \omega_{E_a E_s} \wedge \omega^s.
\end{equation}
Since $\omega_{E_a E_s} =-\omega_{E_s E_a}$ and $\omega^a=E_a^{\flat}$, the equation \eqref{eq_cartan_structure_i_std} follows.
\end{proof}

\begin{remark}
The version of the first structural equation obtained here has the advantage that it can be defined for general vector fields, which are not necessarily from an orthonormal local frame, or a local frame in general. It is well defined even if the {\metricname} becomes degenerate (but {\rstationary}). Of course, at the points where the signature changes we should not expect to have continuity, but on the regions of constant signature the contraction is smooth. If the manifold $(M,g)$ is {\semireg}, the smoothness is ensured even at the points where the {\metricname} changes its signature.
\end{remark}

\section{The second structural equation}
\label{s_cartan_structure_ii}

\subsection{The curvature forms}
\label{s_curvature_form}

\begin{definition}
\label{def_curvature_form}
Let $(M,g)$ be a radical-stationary manifold, let $X,Y,Z,T\in\fivect{M}$ be four vector fields, and $R(X,Y,Z,T)$ the Riemann curvature tensor. Then, the $2$-form
\begin{equation}
\label{eq_curvature_form}
	\Omega_{XY}(Z,T) := R(X,Y,Z,T)
\end{equation}
is called the \textit{Riemann curvature form} associated to the Koszul object $\kosz$, and the vector fields $X,Y$.
In particular, if $(E_a)_{a=1}^n$ is a frame field, we define $\Omega_{ab}$ by
\begin{equation}
\label{eq_curvature_form_frame}
	\Omega_{ab}(Z,T) := \Omega_{E_a E_b}(Z,T).
\end{equation}
\end{definition}

\subsection{The second structural equation}
\label{ss_cartan_structure_ii}

\begin{lemma}
\label{thm_cartan_structure_ii}
Let $(M,g)$ be a radical-stationary manifold, and let $X,Y\in\fivect{M}$ be two vector fields. Then, the equation 
\begin{equation}
\label{eq_cartan_structure_ii}
	\Omega_{XY} = \de\omega_{XY} + \omega_{X\cocontr} \wedge \omega_{Y\cocontr}
\end{equation}
holds, and is called \textit{the second structural equation} determined by the metric $g$.
\end{lemma}
\begin{proof}
From the definition of the exterior derivative it follows
\begin{equation}
\label{eq_cartan_structure_ii_a}
\begin{array}{lll}
\de\omega_{XY}(Z,T) &=& Z\(\omega_{XY}(T)\) - T\(\omega_{XY}(Z)\) - \omega_{XY}([T,Z]) \\
&=& Z\kosz(T,X,Y) - T\kosz(Z,X,Y) - \kosz([T,Z],X,Y).
\end{array}
\end{equation}
On the other hand, 
\begin{equation}
\label{eq_cartan_structure_ii_b}
\begin{array}{lll}
\(\omega_{X\cocontr} \wedge \omega_{Y\cocontr}\)(Z,T) &=& \omega_{X\cocontr}(Z) \omega_{Y\cocontr}(T) - \omega_{X\cocontr}(T) \omega_{Y\cocontr}(Z) \\
&=& \kosz(Z,X,\cocontr) \kosz(T,Y,\cocontr) - \kosz(T,X,\cocontr) \kosz(Z,Y,\cocontr).
\end{array}
\end{equation}
From the equation \eqref{eq_riemann_curvature_tensor_koszul_formula}, it follows
\begin{equation}
\label{eq_cartan_structure_ii_c}
\begin{array}{lll}
	R(X,Y,Z,T)&=& Z \kosz(T,X,Y) - T \kosz(Z,X,Y) - \kosz([Z,T],X,Y)\\
	&& + \kosz(Z,X,\cocontr)\kosz(T,Y,\cocontr) - \kosz(T,X,\cocontr)\kosz(Z,Y,\cocontr),
\end{array}
\end{equation}
and from the identities \eqref{eq_cartan_structure_ii_a}, \eqref{eq_cartan_structure_ii_b} and \eqref{eq_cartan_structure_ii_c} the equation \eqref{eq_cartan_structure_ii} follows.
\end{proof}


%
%
{\bf Acknowledgements.}
Partially supported by Romanian Government grant PN II Idei 1187.

The author wishes to thank Prof. Dr. Constantin Udri\c{s}te and Prof. Dr. Gabriel Pripoae, for stimulating discussions and valuable observations.



\end{document}